\documentclass[preprint]{elsarticle}

\usepackage{mathtools,amsfonts}

\newtheorem{mytheorem}{Theorem}
\newdefinition{mydef}{Definition}
\newtheorem{mylemma}{Lemma}
\newproof{proof}{Proof}

\begin{document}
\begin{frontmatter}
\title{Fa\`a di Bruno's formula for chain differentials}

\author{Daniel~E.~Clark\footnote{Corresponding author.
{\texttt d.e.clark@hw.ac.uk}
School of Engineering and Physical Sciences, Heriot-Watt University, Edinburgh, UK. +44 131 451 3323.}
and~Jeremie~Houssineau}

\begin{abstract}
This paper determines the general formula for describing differentials of composite functions in terms of differentials of their factor functions.
This generalises the formula commonly attributed to Fa\`{a} di Bruno to functions in locally convex topological vector spaces.
The result highlights the general structure of the higher-order chain rule in terms of partitions of the directions.
\end{abstract}

\end{frontmatter}

\section{Introduction}
Mathematicians have investigated formulae for expressing higher-order derivatives of composite functions in terms of derivatives of their factor functions for over 200 years.
These formulae are often attributed to Fa\`{a} di Bruno~\cite{faa2}, though Craik~\cite{Craik} recently highlighted a number of researchers preceding his works, the earliest of which is thought to be by Arbogast~\cite{arbo}.
Despite the fact that the idea of expressing these formulae in terms of derivatives of the factor functions is not new, a number of recent works have appeared on this topic, including those by Hardy~\cite{hardy} and Ma~\cite{ma} on partial derivatives, and an alternative approach was presented by Huang {\it et al.} for Fr\'echet derivatives \cite{huang}, though these forms do not maintain the general structure.
The result presented in this paper has a general multi-linear structure that does not necessarily factorise into the usual product form.

\section{Theory}

In this section we discuss two different forms of differential, the G\^{a}teaux differential~\cite{Gateaux} and the chain differential~\cite{bernhard}.
The chain differential, which is similar to the epiderivative~\cite{AubinFran}, is adopted since it is possible to determine a chain rule, yet is not as restrictive as the Fr\'echet derivative.

\subsection{G\^ateaux differentials}

The following two definitions describe the G\^ateaux differential and its $n^{th}$-order differential.

\begin{mydef}[G\^ateaux differential]
Let $X$ and $Y$ be locally convex topological vector spaces, and
let $\Omega$ be an open subset of $X$ and let $f:\Omega\rightarrow Y$. The G\^ateaux differential at $x\in \Omega$ in the direction $\eta\in X$ is
\begin{equation}
\delta f(x;\eta) = \lim_{\epsilon\to 0} \frac{1}{\epsilon}\left( f(x+\epsilon \eta) - f(x) \right)
\end{equation}
when the limit exists. If $\delta f(x;\eta)$ exists for all $\eta\in X$ then $f$ is G\^ateaux differentiable at $x$. 
The G\^ateaux differential is homogeneous of degree one in $\eta$, so that for all real numbers $\alpha$, $\delta f(x;\alpha \eta)=\alpha\delta f(x;\eta)$.
\end{mydef}

\begin{mydef}[$n^{th}$-order G\^ateaux differential]
The {\it $n^{th}$-order variation} of $f(x)$ in directions $\eta_1,\ldots,\eta_n\in X$ is defined recursively with
\begin{equation*}
\delta^n f\left(x;\eta_1,\ldots,\eta_n\right) =
\delta\left(\delta^{n-1} f\left(x;\eta_1,\ldots,\eta_{n-1}\right); \eta_n\right).
\end{equation*}
\end{mydef}

\subsection{Chain differentials}

Due to the lack of continuity properties of the G\^ateaux differential, further constraints are required in order to derive a chain rule.
Bernhard \cite{bernhard} proposed a new form of G\^ateaux differential defined with sequences, which he called the chain differential, that is not as restrictive as the Fr\'echet derivative though it is still possible to find a chain rule that maintains the general structure.

\begin{mydef}[Chain differential]
The function $f:X\rightarrow Y$, where $X$ and $Y$ are locally convex topological vector spaces, has a {\it chain differential} $\delta f(x;\eta)$ at $x$ in the direction $\eta$ if, for any sequence $\eta_n\rightarrow\eta\in X$, and any sequence of real numbers $\theta_n\rightarrow 0$, it holds that the following limit exists
\begin{equation*}
\delta f(x;\eta) = \lim_{n\rightarrow \infty} \dfrac{1}{\theta_n} \left( f(x+\theta_n\eta_n)-f(x) \right).
\end{equation*}
\end{mydef}

\begin{mylemma}[Chain rule, from \cite{bernhard}, Theorem 1]
\label{lemma:chainRule}
Let $X$, $Y$ and $Z$ be 
 locally convex topological vector spaces, $f : Y \rightarrow Z$ , $g : X\rightarrow Y$ and $g$ and $f$ have chain differentials at $x$ in the direction $\eta$ and at $g(x)$ in the direction $\delta g(x;\eta)$ respectively. Let $h = f \circ g$, then $h$ has a chain differential at $x$ in the direction $\eta$, given by the chain rule
\begin{equation*}
\delta h(x;\eta) = \delta f(g(x); \delta g(x;\eta)).
\end{equation*}
\end{mylemma}

In order to determine the general higher-order chain rule, it is useful to introduce the partial and total chain differentials as follows.

\begin{mydef}[Partial chain differential]
Let $\{X_i\}_{i=1:n}$ and $Y$ be  locally convex topological vector spaces. The function $f:X_1\times\ldots\times X_n \rightarrow Y$ has a {\it partial chain differential} with respect to the $i^{th}$ variable $\delta_i f(x_1,\ldots,x_n;\eta)$ at $(x_1,\ldots,x_n)$ in the direction $\eta$ if, for any sequence $\eta_m\rightarrow\eta\in X$, and any sequence of real numbers $\theta_m\rightarrow 0$, it holds that
the following limit exists
\begin{align*}
&\delta_i f(x_1,\ldots,x_n;\eta) = 
\lim_{m\rightarrow \infty} \dfrac{1}{\theta_m} \left( f(x_1,\ldots,x_i+\theta_m\eta_m,\ldots,x_n)-f(x_1,\ldots,x_n) \right).
\end{align*}
\end{mydef}

\begin{mytheorem}[Total chain differential]
\label{thm:totChainDiff}
Let $\{X_i\}_{i=1:n}$ and $Y$ be  locally convex topological vector spaces. The function $f:X_1\times\ldots\times X_n \rightarrow Y$ has a {\it total chain differential} $\delta f$ at $(x_1,\ldots,x_n)$ if 
\begin{enumerate}
\item all the partial chain differentials exist in a neighbourhood $\Omega \subseteq X_1\times\ldots\times X_n$ of $(x_1,\ldots,x_n)$ and in any direction, and
\item
 $\delta_i f$ is continuous over $\Omega\times X_i$. 
\end{enumerate}
Then for $\underline{\eta} \in X_1\times\ldots\times X_n$ such that $\underline{\eta} = (\eta_1,\ldots,\eta_n)$,
\begin{equation*}
\delta f(x_1,\ldots,x_n;\underline{\eta}) = \sum_{i=1}^n \delta_i f(x_1,\ldots,x_n;\eta_i).
\end{equation*}
\end{mytheorem}

\begin{proof}
The result is proved in the case $n=2$ from which the general case can be straightforwardly deduced:
\begin{align*}
\delta f\left(x,y; (\eta,\xi)\right) 
& = \lim_{r\rightarrow \infty} \theta^{-1}_r\left[ f(x+\theta_r \eta_r,y + \theta_r \xi_r) - f(x,y) \right]\\
& = \lim_{r\rightarrow \infty} \big(\theta^{-1}_r\left[g_1(y + \theta_r\eta_r) - g_1(y) \right] + \theta^{-1}_r\left[g_2(x+\theta_r \eta_r) - g_2(x)\right] \big),
\end{align*}
where we define $g_1(y)$ and $g_2(x)$ as follows:
\begin{equation*}
\left\{
\begin{array}{l}
g_1(y) = f(x+\theta_r \eta_r,y),\\
\\
g_2(x) = f(x,y).
\end{array}
\right.
\end{equation*}

Given $\theta_r \neq 0$, define $h:\mathbb{R} \rightarrow \mathbb{R}$ as $h(t) = g_1(y + t\xi_r)$. From the mean value theorem for real-valued functions, there exists $c_y \in [0,\theta_r]$ such that 
\begin{equation*}
\theta_r^{-1}\left[h(\theta_r) - h(0)\right] = \left.\dfrac{\mathrm{d} h}{\mathrm{d} t}\right|_{t = c_y} = \delta h(c_y;1),
\end{equation*}
which, when replacing $h(t)$ by $g_1(y + t\xi_r)$, can be rewritten
\begin{align*}
\theta_r^{-1}\left[g_1(y + \theta_r\xi_r) - g_1(y)\right] & = \delta \left(g_1(y + c_y\xi_r);1\right)\\
& = \delta g_1(y + c_y \xi_r; \xi_r),
\end{align*}
where Lemma \ref{lemma:chainRule} has been used. 
Similarly for $g_2(x)$, there exists $c_x \in [0,\theta_r]$ such that
\begin{equation*}
\theta^{-1}_r\left[g_2(x+\theta_r \eta_r) - g_2(x)\right] = \delta g_2(x+c_x \eta_r; \eta_r).
\end{equation*}

The last step in the proof is to demonstrate that the limit of the following term
\begin{equation}
\label{eq:lastStep}
\big| \delta g_2(x+c_x \eta_r;\eta_r) + \delta g_1(y + c_y \xi_r;\xi_r)
- \delta_1 f\left(x,y; \eta_r\right) - \delta_2 f\left(x,y; \xi_r\right) \big|,
\end{equation}
is equal to $0$ when $r\rightarrow\infty$.
By the triangle inequality, (\ref{eq:lastStep}) is bounded above by the following summation
\begin{equation*}
\notag\big| \delta g_2(x+c_x \eta_r;\eta_r) - \delta_1 f\left(x,y; \eta_r\right) \big| + \big| \delta g_1(y + c_y \xi_r;\xi_r) - \delta_2 f\left(x,y; \xi_r\right) \big|.
\end{equation*}
Substituting $g_1$ and $g_2$ with $f$, we have
\begin{align}
\big| \delta_1 f\left(x+c_x \eta_r, y; \eta_r\right)  - \delta_1 f\left(x,y; \eta_r\right) \big| + \big| \delta_2 f\left( x ,y + c_y \xi_r; \xi_r\right) - \delta_2 f\left(x,y; \xi_r\right) \big|,
\notag
\end{align}
which tends to $0$ 
when $r\rightarrow\infty$
because of the continuity of $\delta_1 f$ and $\delta_2 f$.

The following result is then proved:
\begin{equation*}
\delta f\left(x,y; (\eta,\xi)\right) = \delta_1 f\left(x,y; \eta\right) + \delta_2 f\left(x,y; \xi\right)
\end{equation*}
which is equivalent to the Proposition $3$ in \cite{bernhard}.
\end{proof}

\begin{mytheorem}[General higher-order chain rule]
Let $X$, $Y$ and $Z$ be  locally convex topological vector spaces.
Assume that $g:X\rightarrow Y$ has higher order chain differentials in any number of directions in the set $\{\eta_1, \ldots, \eta_n\}\in X^n$ and that $f:Y\rightarrow Z$ has higher order chain differentials in any number of directions in the set $\{\delta^m g(x;S_m)\}_{m=1:n}$, $S_m \subseteq \{ \eta_1, \ldots, \eta_n \}$.
Assume additionally that for all $1\leq m \leq n$, $\delta^m f(y;\xi_1,\ldots,\xi_m)$ is continuous on an open set $\Omega \subseteq Y^{m+1}$ and linear with respect to the directions $\xi_1,\ldots,\xi_m$, the $n^{th}$-order variation of composition $f\circ g$ in directions $\eta_1, \ldots, \eta_n$ at point $x\in X$ is given by
\begin{align*}
\delta^n (f\circ g)(x; \eta_1,&\ldots,\eta_n ) = 
\sum_{\pi\in \Pi(\eta_{1:n})} \delta^{|\pi|}f \left( g(x); \xi_{\pi_1}(x),\ldots,\xi_{\pi_{|\pi|}}(x) \right),
\end{align*}
where $\xi_{\omega}(x) = \delta^{|\omega|} g \left( x; \omega_1,\ldots,\omega_{|\omega|} \right)$ is the $|\omega|^{th}$-order chain differential of $g$ in directions $\{\omega_1,\ldots,\omega_{|\omega|}\} \subseteq \{ \eta_1, \ldots, \eta_n \} $. $\Pi(\eta_{1:n})$ represents the set of partitions of the set $\{\eta_1,\ldots,\eta_n\}$ and $|{\pi}|$ denotes the cardinality of the set $\pi$.
\end{mytheorem}

\begin{proof}
Lemma \ref{lemma:chainRule} gives the base case $n=1$.
For the induction step, we apply the differential operator to the case $n$ to give the case $n+1$ and show that it involves a summation over partitions of elements $\eta_1,\ldots,\eta_{n+1}$ in the following way
\begin{equation}
\label{eq:faaProof1}
\delta^{n+1} (f\circ g)(x; \eta_1,\ldots,\eta_{n+1}) = 
\sum_{\pi\in \Pi(\eta_{1:n})} \delta \left( \delta^{|{\pi}|}f \left( g(x); \xi_{\pi_1}(x),\ldots,\xi_{\pi_{|{\pi}|}}(x) \right); \eta_{n+1} \right).
\end{equation}

The main objective in this proof is to calculate the term
\begin{equation}
\label{eq:faaProof2}
\delta \left( \delta^k f \left( g(x); h_1(x),\ldots,h_k(x) \right); \eta \right).
\end{equation}

The additional differentiation with respect to $\eta$ applies to every function on $X$, i.e. to $g$ and to the $h_i$, where $1\leq i\leq k$.
To highlight the structure of this result, we can define a multi-variate function $F$ such that
\begin{align*}
F : Y^{k+1} & \to Z \\
(y_0,\ldots,y_k) &\mapsto \delta^k f(y_0; y_1,\ldots,y_k),
\end{align*}
so that (\ref{eq:faaProof2}) can be rewritten $\delta\left(F(g(x),h_1(x),\ldots,h_k(x)); \underline{\eta}^{(k+1)} \right)$, when denoting $\underline{\eta}^{(k)} = (\eta,\ldots,\eta) \in X^k$. Using Theorem \ref{thm:totChainDiff} we find
\begin{equation}
\label{eq:faaProof3}
\delta\left(F(g(x),h_1(x),\ldots,h_k(x)); \underline{\eta}^{(k+1)} \right) = \sum_{i=1}^{k+1} \delta_i\left(F(g(x),h_1(x),\ldots,h_k(x)); \eta \right).
\end{equation}

From this point, differentiation with respect to $g(x)$ has to be dealt with separately due to the different properties of $F$ with respect to its arguments.
\begin{itemize}
\item Consider the first term of the 
summation in (\ref{eq:faaProof3}):
\begin{equation}
\label{eq:firstTerm}
\delta_1\left(F(g(x),h_1(x),\ldots,h_k(x)); \eta \right).
\end{equation}

Let $\theta_m \to 0$, $\eta^{(m)} \to \eta$, and let $\phi_m$ be defined as
\begin{equation*}
\phi_m(x) = \theta_m^{-1}(g(x+\theta_m \eta^{(m)})-g(x)).
\end{equation*}
Following Theorem $1$ in Bernhard~\cite{bernhard}, we have
\begin{equation}
\label{eq:chainRuleTool}
\left\{
\begin{array}{l}
\phi_m(x) \to \delta g(x;\eta),\\
\\
g(x+\theta_m\eta^{(m)})=g(x) + \theta_m \phi_m(x).
\end{array}
\right.
\end{equation}
One can then rewrite (\ref{eq:firstTerm}) as the limit when $m\rightarrow\infty$ 
of
\begin{equation*}
\theta_m^{-1}\left( F(g(x) + \theta_m\phi_m(x),h_1(x),\ldots,h_k(x)) - F(g(x),h_1(x),\ldots,h_k(x)) \right),
\end{equation*}
which, when taking the limit and using (\ref{eq:chainRuleTool}), can be expressed as
\begin{equation}
\label{eq:finalStep1}
\delta_1\left(F(g(x),h_1(x),\ldots,h_k(x)); \eta \right) = \delta_1 F(g(x),h_1(x),\ldots,h_k(x); \delta g(x,\eta)).
\end{equation}

\item Now consider all the other terms in (\ref{eq:faaProof3}):
\begin{equation}
\label{eq:otherTerm}
\delta_i\left(F(g(x),h_1(x),\ldots,h_k(x)); \eta \right), \qquad 2 < i \leq k+1.
\end{equation}

Let $\theta_m \to 0$ and $\eta^{(m)} \to \eta$. The terms in (\ref{eq:otherTerm}) can be expressed as the limit when $m\rightarrow\infty$  of
\begin{equation*}
\theta_m^{-1}\big( F(g(x),h_1(x),\ldots,h_i(x + \theta_m \eta^{(m)}),\ldots,h_k(x)) - F(g(x),h_1(x),\ldots,h_k(x)) \big).
\end{equation*}
However,
due to the linearity of $F$ with respect to all its arguments except the first, the previous expression can be simplified and written as
\begin{equation*}
F(g(x),h_1(x),\ldots,\theta_m^{-1}\big(h_i(x+\theta_m\eta^{(m)})-h_i(x)\big),\ldots,h_k(x)).
\end{equation*}
Taking the limit, this becomes
\begin{equation}
\label{eq:finalStep2}
\delta_i\left(F(g(x),h_1(x),\ldots,h_k(x)); \eta \right) = F(g(x),h_1(x),\ldots,\delta h_i(x,\eta),\ldots,h_k(x)).
\end{equation}
\end{itemize}

Considering $k = |\pi|$, $\eta=\eta_{n+1}$ and $h_i = \xi_{\pi_i}$ and replacing the results (\ref{eq:finalStep1}) and (\ref{eq:finalStep2}) into (\ref{eq:faaProof1}), we find
\begin{multline*}
\delta^{n+1} (f\circ g)(x; \eta_1,\ldots,\eta_{n+1}) = \\
\sum_{\pi\in \Pi(\eta_{1:n})} \delta^{|{\pi}|+1}f \left( g(x); \xi_{\pi_1}(x),\ldots,\xi_{\pi_{|{\pi}|}}(x), \delta g(x,\eta_{n+1}) \right) \\
+ \sum_{\pi\in \Pi(\eta_{1:n})} \sum_{i=1}^{|\pi|} \delta^{|{\pi}|} f \left( g(x); \xi_{\pi_1}(x),\ldots,\delta\xi_{\pi_i}(x;\eta_{n+1}),\ldots,\xi_{\pi_{|{\pi}|}}(x) \right).
\end{multline*}

Following a similar argument used for the recursion of Stirling numbers of the second kind and their relation to Bell numbers~\cite[p74]{RStanleyV1}, the result above can be viewed as a means of generating all partitions of $n+1$ elements from all partitions of $n$ elements:
The first term corresponds to the creation of a new subset only containing $\eta_{n+1}$, and each term in the second summation appends $\eta_{n+1}$ to one of the existing subset in $\pi \in \Pi(\eta_{1:n})$.
This argument follows similar arguments previously used for ordinary and partial derivatives~\cite{ma,huang,hardy}.
Hence the result is proved by induction.
\end{proof}

\section{Discussion}
It is worth highlighting the structure of the result.
In other forms of chain rule, Fa\`a di Bruno's formula is a sum over partitions of products.
However, in the general form for variational calculus, the outer functional has variations in directions that themselves are differentials of the inner functional.

\subsection{Conclusions}
This study determined a general form of Fa\`a di Bruno's formula based on Bernhard's chain differential.

\bibliography{bibHierSys}

\end{document}